\documentclass{amsart}

\usepackage{multicol}
\usepackage[hidelinks]{hyperref} 
\usepackage{amssymb}
\usepackage[2cell,arrow,matrix]{xy}
\UseHalfTwocells \UseTwocells
\usepackage{MnSymbol}
\usepackage{hyperref}
\usepackage[english]{babel}

\def \Cov{\mathop{\sf Cov}\nolimits}
\def \FEC{\mathop{\sf FEC}\nolimits}
\def \FSet{\mathop{\sf FSet}\nolimits}
\def \FSets{\mathop{\sf FSets}\nolimits}
\def \C{\mathfrak{C}}
\def \E{\mathfrak{E}}
\def \F{\mathfrak{F}}
\def \G{\mathfrak{G}}
\def \cF{\mathcal{F}}
\def \cC{\mathcal{C}}
\def \U{\mathfrak{U}}
\def \Cat{\mathfrak{Cat}}
\def \GCat{\mathsf{GCat}}
\def \colim {\mathop{\sf colim}\nolimits}
\def \lim {\mathop{\sf lim}\nolimits}
\def \pt{\mathsf{pt}}
\def \Sets{\mathsf{Sets}}

\def \Des{\mathsf{Des}}
\def \Hom{\mathsf {Hom}}
\def \Aut{\mathsf {Aut}}

\newtheorem{De}{Definition}
\numberwithin{De}{section}
\newtheorem{Rem1}{Remark}
\newtheorem{Th}[De]{Theorem}
\newtheorem{Pro}[De]{Proposition}
\newtheorem{Le}[De]{Lemma}
\newtheorem{Co}[De]{Corollary}
\newtheoremstyle{}{}{}{}{}{\small \bf \sc}{:}{ }{}
\theoremstyle{}
\newtheorem{Rem}[Rem1]{Remark}

\usepackage{fancyhdr}
\begin{document}

\title{The \'etale Fundamental Groupoid as a Terminal Costack}

\author[I. Pirashvili]{Ilia  Pirashvili}
\email{ilia\_p@ymail.com}
\maketitle

\begin{abstract}
In \cite{top.groupoid} we showed that the fundamental groupoid of a topological space can be defined by the Seifert-van Kampen theorem. This allowed us in effect to give the first axiomatisation of the fundamental groupoid. In this paper, we will prove that the analogue for the \'etale fundamental groupoid of a noetherian scheme $X$ holds as well. The proof here, however, is completely different and involves the study of generalised Galois categories, which could be of intrinsic interest.
\end{abstract}

\section*{Introduction}

Let $X$ be a noetherian scheme and denote by $\FEC(X)$ the site of finite \'etale coverings of $X$. Then we can define the assignment $\Pi_1:Y\mapsto \Pi_1(Y)$, where $Y$ is an \'etale scheme over $X$ and $\Pi_1(Y)$ is the \'etale fundamental groupoid of $Y$. In this paper, we will show that this association, which is a 2-functor, is in actuality a costack over $\FEC(X)$. The costack, the dual notion of a stack, is a rather overlooked construction. However, as we will see, for a covariant 2-functor to be a costack is actually a very natural property, as it merely means that the 2-categorical version of the Seifert-van Kampen theorem holds for every covering. Hence, by proving that the 2-functor $\Pi_1$ is a costack, we have proven that a slightly reformulated version of the Seifert-van Kampen theorem holds for the \'etale fundamental groupoid.

But more is true. As we will show in this paper, $\Pi_1$ is not just a costack over $X$, but indeed the 2-terminal costack. This demonstrates that the \'etale fundamental groupoid is indeed defined by the Seifert-van Kampen theorem. This result agrees with the one obtained in \cite{top.groupoid}, where we showed that the topological fundamental groupoid was a 2-terminal costack. Hence we are able to talk about the fundamental groupoid of a site, which could be of independent interest. Indeed, although not discussed in this paper, another use of this axiomatisation is that we can now also talk about other fundamental objects. For example, we could define the fundamental groupoid scheme of a site $\bf{T}$ to be the 2-terminal objects in the 2-category of costacks over $\bf{T}$ with values in groupoid schemes.
\newline

This paper was written as part of my PhD thesis at the University of Leicester under the supervision of Dr. Frank Neumann. He introduced me to the fundamental groupoid and inspired much of this paper, for which I would like to thank him.

\section{Preliminaries}
In this section we will fix some basic terminology. Recall that a 2-category is a category enriched in categories. A category $\G$ is said to be a \emph{groupoid}\index{Groupoid!} if every morphism is an isomorphism. 

\begin{De} Let $\G$ be a groupoid such that for every $a,b\in \G$, $\Hom(a,b)$ is endowed with the profinite topology and further the composition and inverse maps are continuous. We say that $\G$ is a finitely connected profinite groupoid, if it is equivalent to the 2-coproduct of finitely many connected groupoids coming from profinite groups.
\end{De}

Note that this is not the best way to define a profinite groupoid in the general case, since it does not deal with the topology on its connected components. One should define it as simply the filtered 2-limit of finite groupoids endowed with the profinite topology. But for the purposes of this paper, this simpler definition is adequate.

\subsection{2-Limits and 2-Colimits}\label{sec.2-lim.2-colim}

Just like limits and colimits are of fundamental importance in category theory, so too are 2-limits and 2-colimits in 2-category theory. In this subsection, we will talk a little about the 2-limit as well as the filtered 2-colimit of 2-functors. 
\newline

Let $I$ be a category and $\F:I\rightarrow \Cat$ be  a covariant 2-functor from the category $I$ to the 2-category $\Cat$ of categories. For an element $i\in I$ we let $\F_i$ be the value of $\F$ at $i$. For a morphism $\psi:i\rightarrow j$ we let $\psi_*:\F_i\to \F_j$ be the induced functor. For any $i\xrightarrow{\psi} j\xrightarrow{\nu} k$, one has the natural transformation $\mu_{\psi,\nu}:\nu_*\psi_*\rightarrow (\nu\psi)_*$ satisfying the coherence condition. We can now give the definitions:

Objects of the category 2-$\lim\limits_i\F_i$ (or simply 2-$\lim \F$) are collections $(x_i,\xi_\psi)$, where $x_i$ is an object of $\F_i$, while $\xi_\psi:\psi_*(x_j)\to x_i$ for $i\leq j$, is an isomorphism of the category $\F_j$ satisfying the 1-cocycle condition. That is, for any $i\xrightarrow{\psi} j\xrightarrow{\nu} k$, the diagram
$$\xymatrix{\nu_*(\psi_*(x_i)) \  \  \ar[r]^{\nu_*(\xi_{\psi})}\ar[d]^{\mu_{\psi,\nu}} & \nu_*(x_j)\ar[d]^{\xi_{\nu}} \\ 
	          (\nu\psi)_*(x_i)\ar[r]_{\xi_{\psi\nu}} & x_k }$$ 
commutes. A morphism from $(x_i,\xi_{\psi})$ to $(y_i,\eta_{\psi})$ is  a collection $(f_i)$, where $f_i:x_i\rightarrow y_i$ is a morphism of $\F_i$ such that for any $\psi:i\rightarrow j$, the following 
$$\xymatrix{\psi_*(x_i)\ar[r]^{\xi_{\psi}}\ar[d]_{\psi_*(f_i)}& x_j\ar[d]^{f_j} \\ 
\psi_*(y_i)\ar[r]_{\eta_{\psi}}& y_j}$$
 is a commutative diagram. 
\newline

The dual notion $2$-$\colim\limits_i \F_i$ is a category together with a family of functors
 $$\alpha_i:\F_i\to 2\text{-}\colim\limits_i \F_i$$
 and natural transformations $\lambda_{ij}:\alpha_j\psi_{ij}\Longrightarrow \alpha_i$, $i\leq j$,  satisfying  the $1$-cocycle condition: For any $i\leq j\leq k$,
 $$\lambda_{ik}=\lambda_{ij}\circ (\lambda_{jk}\star\psi_{ij}).$$
Furthermore, one requires that for any category ${\sf C}$, the canonical functor
$$\kappa:\Hom_{\Cat}(2\text{-}\colim\limits_i\F_i,\mathsf{C})\to 2\text{-}\lim\limits_i\Hom_{\Cat}(\F_i,{\sf C})$$
is an equivalence of categories. Here the functor $\kappa$ is given by $\kappa(\chi)=(\chi\circ  \alpha_i, \chi_i\star \lambda_{ij})$. It is well-known that $2$-$\colim$ exists and is unique up to a unique equivalence of categories (see \cite[pp. 192-193] {br}).

\section{Stacks and Costacks}\label{2-mathematics}

\subsection{Stacks}\label{stacksection}

This subsection follows \cite{I.Mo} closely, even on notation. Let $X$ be a site and $\F:X^{op}\rightarrow\Cat$ a 2-functor where $\Cat$ is the 2-category of categories. This is called a \emph{fibered category}\index{Fibered! Category} over $X$. It should be noted to avoid any confusion that this is sometimes called a prestack. If we have two fibered categories over $X$, then a morphism between them is called a \emph{fibered functor}\index{Fibered! Functor}. 

Let $X$ be a site and $\F:X^{op}\rightarrow\Cat$ a 2-functor. Let $U$ be an object in $X$ and $\U=\{U_i\rightarrow U\}$ a covering of $U$. Then we can consider the following diagram: 
$$\xymatrix{\bigsqcap\limits_{i\in I}\F(U_i)\ar@<-.5ex>[r] \ar@<.5ex>[r]& \bigsqcap\limits_{i,j\in I}\F(U_{ij})\ar@<-.7ex>[r]\ar[r]\ar@<.7ex>[r] & \bigsqcap\limits_{i,j,k\in I}\F(U_{ijk}).}$$
Here $U_{ij}=U_i\times_{U}U_j$ and $U_{ijk}=U_i\times_{U}U_j\times_{U}U_k$.
We denote its 2-limit by $2$-$\lim(\U,\F)$. Note that the 2-limit in this case is usually called the \emph{descent data} and denoted by $\Des(\U,\F)$. However, working with 2-limits makes it much clearer what a costack should be, hence we will keep using the above notation. 

\begin{De}[Stack]\label{stack} A fibered category $\F$ over $X$ is called a \emph{stack}\index{Stack} if for all objects $U$ of $X$ and for all coverings $\mathfrak{U}$ of $U$, the functor $\F(U)\rightarrow 2$-$\lim(\mathfrak{U},\F)$ is an equivalence of categories.
\end{De}

Just like for sheaves, we can talk about the associated stack of a 2-functor.

\begin{De}\label{uniquestack} Let $\F$ be a fibered category over a site $X$. Then $\hat{\F}$ is the associated stack of $\F$ if for every stack $\G$ over $X$, we have an equivalence of categories 
$$\Hom_{\mathsf{Fb}}(\F,\G)\cong\Hom_{\mathsf{St}}(\hat{\F},\G)$$
where $\mathsf{Fb}$ denotes the 2-category of fibered categories and $\mathsf{St}$ the 2-category of stacks.
\end{De}

\begin{Pro}[Direct Stackification]\label{direct} Let $\F:X^{op}\rightarrow\Cat$ be a 2-functor. Define $\F'(U):=2$-$\colim_\U(2$-$\lim(\U,\F))$. In other words, we take the filtered 2-colimits of the categories $2$-$\lim(\U,\F)$ over all coverings $\U$ of $U$. After iterating it 3 times we get the stackification. That is to say, we have 
$$\hat{\F}=\F'''(U)=2\text{-}\colim_\U(2\text{-}\lim(\U,\F'')).$$ 
\end{Pro}

For the proof of this theorem, see \cite[Theorem 3.8]{Street}.

\subsection{Costacks}

Let $X$ be a site and $\F:X\rightarrow\Cat$ a 2-functor where $\Cat$ is the 2-category of categories. Dually to the above subsection, we call $\F$ a \emph{cofibered category}\index{Cofibered category} over $X$. 

Take an object $U$ of $X$ and a covering $\U=\{U_i\rightarrow U\}$. Then we can consider the following diagram: 
$$\xymatrix{\bigsqcap\limits_{i\in I}\F(U_i)& \bigsqcap\limits_{i,j\in I}\F(U_{ij})\ar@<-.5ex>[l] \ar@<.5ex>[l]& \bigsqcap\limits_{i,j,k\in I}\F(U_{ijk})\ar@<-.7ex>[l]\ar[l]\ar@<.7ex>[l]}$$
where $U_{ij}=U_i\times_{U}U_j$ and $U_{ijk}=U_i\times_{U}U_j\times_{U}U_k$.
We denote its 2-colimit by $2$-$\colim(\U,\F)$. 

\begin{De}[Costack]\label{cost} A cofibered category $\F$ over $X$ is called a \emph{costack}\index{Costack} if for all objects $U$ of $X$ and for all coverings $\mathfrak{U}$ of $U$, the functor $\F(U)\leftarrow 2$-$\colim(\mathfrak{U},\F)$ is an equivalence of categories.
\end{De}

Alternatively, we can define a costack using stacks. Namely, $\F$ is a costack if for every category $C$, the assignment $U\mapsto \Hom_{\C}(\F(U),C)$ is a stack. It is clear that this is equivalent to the above since $\Hom(A,-)$ is left exact. However, it should be noted that since it is not (in general) right exact, the duality between stacks and costacks breaks down here. Namely, it would not be sufficient to check that $U\mapsto\Hom_\Cat(\F(U),-)$ is a costack for $\F$ to be a stack.

\begin{De}[Uniqueness Property]\label{uniquecostack} Let $\F$ be a cofibered category over a site $X$. Then we will say that $\F$ has an associated costack $\hat{\F}$ if for every costack $\G$ over $X$, we have an equivalence of categories 
$$\Hom_{\mathsf{Cofb}}(\G,\F)\cong\Hom_{\mathsf{Cost}}(\G,\hat{\F})$$
where $\mathsf{Cofb}$ denotes the 2-category of cofibered categories and $\mathsf{Cost}$ the 2-category of costacks.
\end{De}

Note that if our category took values in groupoids, then it would be enough to check it for every groupoid. Unfortunately, unlike for stacks, it is not known whether every cofibered category has an associated costack. However, if it exists, it is clearly unique in the 2-categorical sense.

\section{Galois Categories} \label{galoischapter}

In this section we will generalise the classical (connected) Galois categories to the finitely connected case. Our main interest in them is to study the \'etale fundamental groupoid of a finitely connected, noetherian scheme. 

In the first subsection, we will show that the analogue of Grothendieck's classical result holds, and that a finitely connected Galois category is equivalent to the category of $\G$-$\FSets$, where $\G$ is a finitely connected, profinite groupoid and $\FSets$ denotes the category of finite sets.

In the next subsection, we will talk about the whole 2-category of such Galois categories, which includes the morphisms and 2-morphisms. We sharpen the result above, by proving that there is a 2-equivalence between the 2-category of finitely connected categories and the 2-category of finitely connected profinite groupoids. While there are many generalisations of Galois categories, some of which are no doubt more general than ours, this reformulation of the classical theorem seems to be new. 

Indeed, it will follow that an even more general formulation is true, namely Corollary \ref{Galoisstackequivalence}.


\subsection{Finitely Connected Galois Categories}

In order to give the main definition, recall that a morphism $u:A\to B$ of a category $\mathcal{C}$ is an \emph{epimorphism}\index{Morphism! Epimorphism} (resp. \emph{monomorphism}\index{Morphism! Monomorphism}) if for any object $X$ the induced map $\Hom_{\mathcal{C}}(B,X)\to \Hom_{\mathcal{C}}(A,X)$ (resp. $\Hom_{\mathcal{C}}(X,A)\to \Hom_{\mathcal{C}}(X,B))$)  is injective. Moreover, an epimorphism $u$ is called a \emph{strict epimorphism}\index{Morphism! Strict epimorphism} if the pull-back 
$$\xymatrix{A\times _B A\ar@{-->}[r]^{p_1} \ar@{-->}[d]^{p_2} & A\ar[d]^u\\A\ar[r]^u & B}$$ 
exists and $B$ is the coequaliser of the diagram: 
$$\xymatrix{ A\times_B A \ar@<0.5ex>[r]^{\ \ \ p_1} \ar@<-0.5ex>[r]_{\ \ \ p_2} & A\ar[r]^{u} & B. }$$

We give the following definition of a finitely connected Galois category which differs from the standard definition of a Galois category (see for example \cite{AGGT}).

\begin{De}\label{f-c G c} A {\sf (finitely-connected) Galois category}\index{Galois Category!} is a category $\mathcal{C}$ together with a set of covariant functors $\{\mathcal{F}_j:\mathcal{C}\rightarrow \FSets\}_{j\in J}$, satisfying the following axioms: 
\begin{enumerate}
\item Finite limits exist in $\mathcal{C}$. 
\item Finite colimits exist in $\mathcal{C}$. 
\item Any morphism $u:Y\rightarrow X$ in $\mathcal{C}$ factors as $Y\xrightarrow{u'}X'\xrightarrow{u''}X$, where $u'$ is a strict epimorphism and $u''$ is a monomorphism and there is an isomorphism $v:X'\coprod X''\to X$ such that $u''=vi_1$, where $i_1:X'\to X'\coprod X''$ is the standard inclusion.
\item Every $\mathcal{F}_j$ is right exact, i.e. $\mathcal{F}_j$ respects finite colimits.
\item Every $\mathcal{F}_j$ is left exact,  i.e. $\mathcal{F}_j$ respects finite limits.
\item There exists a finite subset $I\subset J$ such that any $\{u:Y\rightarrow X\}$ in $\mathcal{C}$ is an isomorphism if and only if $\mathcal{F}_i(u)$ is an isomorphism for all $i\in I$.
\end{enumerate}
\end{De} 

If $I$ can be chosen to be a one element set, then $\mathcal{C}$ is called \emph{connected}\index{Galois Category! Connected}. This is clearly equivalent to the standard definition of a Galois category. Several facts (Lemma 2.6 i), ii); Prop 3.1; Prop. 3.2 (1), (3) i), iii)) proven in \cite{AGGT} have immediate generalisations in our situation. To state these statements, recall that an object $X$ is called \emph{connected}\index{Object! Connected} if $X\not =0$ and  for any decomposition $X=Y\coprod Z$ one has $X=0$ or $Y=0$. Here $0$ denotes the initial object. 

\begin{Pro} \label{2612} Let $\mathcal{C}$ be a finitely connected Galois category. Then the following properties hold:
\begin{itemize} 
\item[i)] A morphism $u$ is a monomorphism (resp.  strong epimorphism) if and only if for all $i\in I$, the map  $\mathcal{F}_i(u)$ is injective (resp. surjective). A morphism is an isomorphism if and only if it is a monomorphism and a strong epimorphism. 
\item[ii)] An object $X$ is initial (resp. terminal) provided  for all $i\in I$, the set $\mathcal{F}_i(X)=\emptyset$ (resp. $\mathcal{F}(X)=*$). Here $*$ denotes the singleton.
\item[iii)] Any object $X$ has a unique decomposition $X=U_1\coprod \cdots \coprod U_k$, where the $U_i$ are connected. 
\item[iv)] If $U$ and $V$ are connected, then any morphism $U\to V$ is a strong epimorphism. In particular, any endomorphism $U\to U$ is an automorphism.
\item[v)] If $U$ is connected, then for any objects $A_1\cdots A_m$ the natural map
$$\coprod_{i=1}^m \Hom(U, A_i)\to \Hom(U, \coprod_{i=1}^m A_i)$$
is a bijection.
\end{itemize}
\end{Pro}

The proof is the same as for the connected case. (see \cite{AGGT}).

\begin{Le}\label{2712}  Let $\mathcal{C}$ be a finitely-connected Galois category and $t=e_1\coprod \cdots \coprod  e_d$ be a decomposition of the terminal object as a coproduct of connected objects. Then for each $1\leq i\leq d$  one has (after re-indexing)  $\mathcal{F}_i(e_i)=*$ and $\mathcal{F}_i(e_j)=\emptyset$, $j\not= i, \  1\leq i,j\leq d$.
\end{Le}

\begin{proof}  Let $1\leq i\leq d$. Since $0\to e_i$ is not an isomorphism, there exists at least one $\mathcal{F}$ such that $\mathcal{F}(e_i)\not =\emptyset$. After reindexing we can assume that $\mathcal{F}=\mathcal{F}_i$. Since $\mathcal{F}_i(t)=*$ and $\mathcal{F}_i$ respects coproducts we see that 
$$\mathcal{F}_i(e_1)\coprod \cdots \coprod \mathcal{F}_i(e_d)=*.$$
So all terms except $\mathcal{F}_i(e_i)$ are empty sets and the result follows.
\end{proof}

Now we are in the position to prove the following result.

\begin{Le}[Main Lemma]\label{G.MainLemma} Let $\mathcal{C}$ be a finitely-connected Galois category. If
$$t=e_1\coprod \cdots \coprod  e_d$$
is the decomposition of the terminal object $t$ as a coproduct of connected objects, then there is an equivalence of categories
$$\mathcal{C}\cong \mathcal{C}_1\times \cdots \times \mathcal{C}_d$$
where $C_i$ , $1\leq i\leq d$ is the following full subcategory of $C$:
$$\mathcal{C}_i=\{X\in \mathcal{C}| \ \mathcal{F}_j(X)=\emptyset, j\not =i, 1\leq j\leq d\}.$$
Furthermore, the pair $(\mathcal{C}_i,\mathcal{F}_i)$ is a connected Galois category and for any element $k\in I$ the functor $\mathcal{F}_k$ is isomorphic to exactly one of the functors $\mathcal{F}_1,\cdots , \mathcal{F}_d$.
\end{Le}

\begin{proof} We proceed by induction on $d$. Assume $d=1$. Thus $t$ is connected. In this case $\mathcal{C}_1=\mathcal{C}$. Take any of $\mathcal{F}_i$ and call it $\mathcal{F}$. First we show that $\mathcal{F}$ reflects isomorphisms, meaning that if $v$ is a morphism, such that $\mathcal{F}(v)$ is an isomorphism, then $v$ is an isomorphism. If $U$ is connected, then $U\to t$ is a strict epimorphism thanks to Proposition \ref{2612} iv). It follows that $\mathcal{F}(U)\to \mathcal{F}(t)=*$ is a strict epimorphism. Thus $\mathcal{F}(U)\not=\emptyset$. Since any object is a coproduct of connected ones and $\mathcal{F}$ respects coproducts, it follows that if $A$ is not an initial object, then $\mathcal{F}(A)\not =\emptyset.$
Assume $u:A\to B$ is a monomorphism, such that $\mathcal{F}(u)$ is an isomorphism. Then $B\cong A\coprod \mathcal{C}$. Hence $\mathcal{F}(\mathcal{C})=\emptyset$, so $\mathcal{C}=0$ and $u$ is an isomorphism. 

Now  let $v:A\to B$ be a general morphism, such that $\mathcal{F}(v)$ is an isomorphism. Consider the following commutative diagram
$$\xymatrix{A\ar[drr]^{id}\ar[dr]_{\delta}\ar[ddr]_{id}&&\\ & A\times_BA\ar[r]\ar[d]&A\ \ar[d]_v\\&A\ar[r]^v& B .}$$
Here $\delta$ is the diagonal map and hence a monomorphism. Apply $\mathcal{F}$ to this diagram and use the fact that $\mathcal{F}$ preserves pulbacks and $\mathcal{F}(v)$ is an isomorphism. We obtain that $\mathcal{F}(\delta)$ is an isomorphism. Thus $\delta$ is an isomorphism. It follows that $v$ is a monomorphism (thanks to \cite[Lemma 2.4]{AGGT}) and hence an isomorphism. 
Thus $\mathcal{F}$ reflects isomorphisms and $(\mathcal{C},\mathcal{F})$ is a connected Galois category. By \cite[Theorem 2.8]{AGGT} any other $\mathcal{F}_i$ is isomorphic to $\mathcal{F}$. Hence we have proven the lemma for $d=1$.

Assume now that $d>1$. Thanks to Lemma \ref{2712}, we have $\mathcal{F}_i(e_j)=\emptyset$ for all $j\not =i$ and $\mathcal{F}_i(e_i)=*$, $1\leq i,j\leq d$. One easily sees that for each $1\leq i\leq d$ the subcategory $\mathcal{C}_i$ is closed under finite limits and colimits. Assume $X=\coprod_{j=1}^k U_j$ is a decomposition as a coproduct of connected objects. 

\underline{Claim 1}: We have $X\in \mathcal{C}_i$ if and only if $U_1,\cdots, U_k\in \mathcal{C}_i$. In fact if  $U_1,\cdots, U_k\in \mathcal{C}_i$, then for any $j\not =i$, $1\leq j\leq d$ one has $\mathcal{F}_j(U_1)=\cdots =\mathcal{F}_j(U_k)=\emptyset$. Thus
$\mathcal{F}_j(X)=\mathcal{F}_j(U_1)\coprod\cdots\coprod \mathcal{F}_j(U_k)=\emptyset.$ Hence $X\in \mathcal{C}_i$. Conversely, if $X\in \mathcal{C}_i$, then $$\emptyset =\mathcal{F}_j(X)=\mathcal{F}_j(U_1)\coprod\cdots\coprod \mathcal{F}_j(U_k).$$
Thus $\mathcal{F}_j(U_1)=\cdots= \mathcal{F}_j(U_k)=\emptyset$ and $U_1,\cdots U_k\in \mathcal{C}_i$.

\underline{Claim 2}: The object $e_i$ is a terminal object in the category $\mathcal{C}_i$. So, we have to prove that the set $\Hom(X,e_i)$ is a singleton provided $X\in \mathcal{C}_i$. By the first claim it is enough to assume that $X$ is connected. According to Proposition \ref{2612} v) 
 $$*=\Hom(X,t)=\Hom(X,e_1)\coprod \cdots \coprod \Hom(X,e_d).$$
 So the set $\Hom(X,e_i)$ has at most one element. To show that it has exactly one element, we need to show that $\Hom(X,e_j)=\emptyset$ for $j\not =i$. In fact, assume there is a morphism $X\to e_j$. Since both objects are connected, this map must be a strict epimorphism. This implies that $\emptyset =\mathcal{F}_j (X)\to \mathcal{F}_j(e_j)=*$ is surjective and hence a contradiction. Thus our second claim is proven. It follows from the case $d=1$, that the pair $(\mathcal{C}_i,\mathcal{F}_i)$ is a connected Galois category.

\underline{Claim 3}: Our third claim is that if $i\not =j$, then for any objects $0\not =X\in \mathcal{C}_i$ and $Y\in \mathcal{C}_j$ one has $\Hom(X,Y)=\emptyset$. In fact, since $e_j$ is terminal in $\mathcal{C}_j$ there exists a unique morphism $Y\to e_j$. Thus it suffices to show that $\Hom(X,e_j)=\emptyset$, but this was shown in the proof of Claim 2. 
\newline 

Define the functor
$$\xi:\mathcal{C}_1\times \cdots \times \mathcal{C}_d\to \mathcal{C}$$
by $\xi(X_1,\cdots,X_d)=X_1\coprod \cdots \coprod X_d$.
We will show that the functor $\xi$ is an equivalence of categories. Take an object $X\in \mathcal{C}$ and consider the pull-back
$$\xymatrix{X_i\ar[d]\ar[r]&X\ar[d]\\e_i\ar[r]&t.}$$
\underline{Claim 4}: We want to show that $X_i\in \mathcal{C}_i$. In fact, take any $j\not =i$. Since the functor $\mathcal{F}_j$ respects pullbacks we obtain a diagram of sets
$$\xymatrix{\mathcal{F}_j(X_i)\ar[d]\ar[r]&\mathcal{F}_j(X)\ar[d] \ \\ 
\emptyset \ar[r]& \mathcal{F}_j(t).}$$
It follows that $\mathcal{F}_j(X_i)=\emptyset$ and thus $X_i\in \mathcal{C}_i$. Moreover, the natural morphism $X_1\coprod X_d\to X$ is an isomorphism, because every $\mathcal{F}_i$ takes it to an isomorphism. It follows that the functor $\xi$ is essentially surjective. It remains to show that the functor $\xi$ is full and faithful. Take the objects $X_i,Y_i\in \mathcal{C}_i$, $i=1,\cdots,d$. We have 
$$\Hom(\coprod_{i=1}^d,\coprod_{i=1}^dY_i)=\prod_{i=1}^d\Hom(X_i, Y_1\coprod \cdots \coprod Y_d)$$
Thus it remains to show that for any object $Z\in \mathcal{C}_i$ one has 
$$\Hom(Z, Y_1\coprod \cdots \coprod Y_d)=\Hom(Z,Y_i)$$
If $Z$ is connected this follows from Proposition \ref{2612} v) and Claim 3. For the general case we decompose $Z=Z_1\coprod \cdots \coprod Z_k$. 

Then  we  have
\begin{align*}
\Hom(Z, Y_1\coprod \cdots \coprod Y_d)&=\Hom(\coprod_{j=1}^kZ_j, Y_1\coprod \cdots \coprod Y_d)\\ 
&=\prod_{j=1}^k\Hom(Z_j, Y_1\coprod \cdots \coprod Y_d)\\ &=\prod_j\Hom(Z_j,Y_i)\\
&=\Hom(\coprod Z_j,Y_i)\\ 
&=\Hom(Z,Y_i)
\end{align*}
Hence $\xi$ is an equivalence of categories.
\end{proof}

Recall the following easy but important fact:

\begin{Pro}\label{homrep} Let $\G$ be a connected groupoid and let $x\in\G$. Then we have an equivalence of categories 
$$\Hom_{\Cat}(\G,\Sets)\cong \Aut(x)\text{-}\mathsf{Sets}$$
where $\Aut(x)\text{-}\mathsf{Sets}$ denotes the category of $\Aut(x)$-Sets.
\end{Pro}

If we changed $\Sets$ with $\FSets$, the category of finite sets, the above would still hold. Hence we can now generalise group actions to groupoid actions and for a groupoid $\G$, we write $\G$-$\FSets$ or $\Hom_\Cat(\G,\FSets)$. It should be emphasised that we use two different notations for the exact same category. The second notation will  mainly be used in calculations. If, however, $\G$ were a profinite groupoid, then we would only consider the continuous actions. For simplicity though, by abuse of notations, we will still refer to it as $\Hom_\Cat(\G,\FSet)$ or $\G$-$\FSets$. 

\begin{Co}\label{G.essentialsurj} Let $\{\mathcal{F}_i:\mathcal{C}\rightarrow \FSets\}_{i\in I}$ be a finitely connected Galois Category. Then it is equivalent to $\G$-$\FSets$, where $\G$ is a finitely connected profinite groupoid.
\end{Co}

\begin{proof} As proven in the above lemma (\ref{G.MainLemma}), 
$$\{\mathcal{F}_i:\mathcal{C}\rightarrow \FSets\}_{i\in I}\cong\{\mathcal{F}_j:\prod_{j'\in J} \mathcal{C}_{j'}\rightarrow \FSets\}_{j\in J}$$ 
such that $J$ is a finite set and $\mathcal{F}_j(C_k)=\emptyset$ for $k\neq j$. Hence our Galois category is equivalent to $\mathcal{F}_j:\prod_{j\in J}\mathcal{C}_j\rightarrow \FSets$. Again by the above lemma, we know that for each $j\in J$, the functor $\mathcal{F}_j:\mathcal{C}_j\rightarrow \FSets$ is a connected Galois category and hence, using Proposition \ref{homrep}, we know that it is equivalent to $\G$-$\FSets$ where $\G$ is a connected, profinite groupoid. The result now follows from the fact that 
$$\prod_{j\in J}\Hom_{\Cat}(\G_j,\FSets)\cong \Hom_{\Cat}(\coprod_{j\in J}\G_j,\FSets).$$
\end{proof}

\subsection{The 2-category of Galois Categories}

\begin{De} Let $\{\mathcal{F}_i:\mathcal{C}\rightarrow \FSets\}_{i\in I}$ and $\{\mathcal{G}_j:\mathcal{D}\rightarrow \FSets\}_{j\in J}$ be two Galois categories. A morphism of Galois categories consists of a map $f:J\rightarrow I$, a functor $\varphi:\mathcal{C}\rightarrow \mathcal{D}$ preserving finite limits and finite colimits and a collection of isomorphisms $\lambda_{j,\varphi},j\in J$, as given in the following diagram 
$$\xymatrix{ \mathcal{D}\ar[rd]_{\mathcal{G}_j} & \dtwocell\omit{\lambda_{j,\varphi}}& \mathcal{C}\ \ar[ll]_{\varphi}\ar[dl]^{\mathcal{F}_f(j)}  \\
& {\FSets} & . }$$ 
We will refer to it as $\{(f,\varphi,\lambda_{j,\varphi}):\mathcal{F}_{f(j)}\rightarrow \mathcal{G}_j\}$. 
\end{De} 

To define composition, we need to define the composition of the $\lambda_{j,\varphi}$'s. So say we now have 
$$\xymatrix{ \mathcal{E}\ar[rrrrdd]_{\mathcal{H}_k}& & \drtwocell\omit{\ \ \ \ \lambda_{k,\phi}} & & \mathcal{D}\ar[dd]^{\mathcal{G}_{f(k)}}\ar[llll]^{\phi} & & \dltwocell\omit{\lambda_{k,\varphi}} & & \mathcal{C}\ \ar[llll]^{\varphi}\ar[ddllll]^{\ \ \ \ \mathcal{F}_g(i)} & &  \\ 
& & & & & & & & \\ 
& & & & {\FSets} & & & & .}$$ 
Define $\lambda_{k,\phi}\circ\lambda_{k,\varphi}(x)=\lambda_{k,\phi}(\varphi(x))\circ\lambda_{k,\varphi}(x).$ In more detail we have 
$$\lambda_{k,\phi}\circ\lambda_{k,\varphi}(x):\mathcal{F}(x)\xrightarrow{\lambda_{k,\varphi}(x)}\mathcal{G}(\varphi(x))\xrightarrow{\lambda_{k,\phi}(\varphi(x))}\mathcal{E}(\phi\circ\varphi(x)).$$

It is easily verified that the above construction is strictly associative.

\begin{De} A 2-morphism between $\{(f,\varphi,\lambda_{j,\varphi}):\mathcal{F}_{f(j)}\rightarrow \mathcal{G}_j\}$ and \newline $\{(f,\phi,\lambda_{j,\phi}):\mathcal{F}_{f(j)}\rightarrow \mathcal{G}_j\}$ is a natural transformation 
$$\xymatrix{ \mathcal{C}\ar@<-1.5ex>[r]_{\phi} \ar@<1.5ex>[r]^{\varphi}\rtwocell\omit_{\ \ \ \ \ \zeta} & \mathcal{D} }$$
such that additionally the following diagram 
$$\xymatrix{\mathcal{F}(x)\ar[r]^{\lambda_{j,\varphi}(x)}\ar[dr]_{\lambda_{j,\phi}(x)} & \mathcal{G}(\varphi(x))\ar[d]^{\zeta(x)} \\
	           & \mathcal{G}(\phi(x)) }$$
commutes for all $j$ and all $x$.
\end{De}

This shows that we can talk about the (strict) 2-category of Galois categories. We will denote it by $\GCat$. 

\begin{Le}\label{hshjsjk} Let $\{\mathcal{F}_i:\mathcal{C}_1\times\cdots\times \mathcal{C}_n\rightarrow \FSets\}_{i\in I}$ be a finitely connected Galois category and $\mathcal{G}:\mathcal{D}\rightarrow \FSets$ a connected Galois category. 

Let $\mathcal{A}:\mathcal{C_1}\times\cdots\times \mathcal{C}_n\rightarrow \mathcal{D}$ be a functor between the Galois categories preserving finite limits and finite colimits. Then there exists an $i\in I$ and $\mathcal{A}_i:\hat{\mathcal{C}_i}\rightarrow \hat{\mathcal{D}}$, such that $\mathcal{A}=\mathcal{A_i}\circ p_i$, where $p_i:\mathcal{C}_1\times\cdots\times \mathcal{C}_n\rightarrow \mathcal{C}_i$ is the $i$-th projection.
\end{Le}

\begin{proof} Take the terminal object $t\in \mathcal{C}$. As in Lemma \ref{G.MainLemma}, $t=\coprod_{i\in I}e_i$, where every $e_i$ is connected. Since $\mathcal{D}$ is connected, its terminal object is connected. Since $\mathcal{A}$ respects finite limits and colimits there exists $i\in I$ such that $\mathcal{A}(e_i)=\star$ and $\mathcal{A}(e_j)=\emptyset$ for all $j\neq i$. Recall now the equivalence $\mathcal{C}\cong \mathcal{C}_1\times\cdots\times \mathcal{C}_n$ as constructed in Lemma \ref{G.MainLemma}. For every $X\in \mathcal{C}$ we have $X\cong \coprod_{i}^n X_i$ where $X_i$ is the pullback of the diagram:
$$\xymatrix{ & X\ar[d] \\ 
		e_i\ar[r] & t=\coprod_{i\in I}e_i.}$$ 
Since $\mathcal{A}$ respects pullbacks, for all $j\neq i$ we have that $\mathcal{A}(X_j)$ is the pullback of 
$$\xymatrix{ & \mathcal{A}(X)\ar[d] \\ 
		\mathcal{A}(e_j)\ar[r] & \mathcal{A}(t)}$$ 
and since $Y\rightarrow \emptyset$ implies that $Y$ is the empty set, we get that $\mathcal{A}(X_j)=\emptyset$. So $\mathcal{A}(X)=\coprod\limits_j \mathcal{A}(X_j)=\mathcal{A}(X_i)$. Hence $\mathcal{A}$ factors through the projection $\mathcal{C}_1\times\cdots\times \mathcal{C}_n\rightarrow \mathcal{C}_i$.
\end{proof}

\begin{Co}\label{G.faithfullyflat} Let $\G_i$-s be finitely connected profinite groupoids indexed by a finite set $I$ and $\mathfrak{H}$ be a connected profinite groupoid. Then we have an equivalence of categories: 
$$\Hom_\GCat(\Hom_\Cat(\coprod_{i\in I}\G_i,\FSets),\Hom_\Cat(\mathfrak{H},\FSets))\cong $$
$$\cong\coprod_{i\in I}\Hom_{\GCat}(\Hom_\Cat(\G_i,\FSets),\Hom_\Cat(\mathfrak{H},\FSets)).$$
\end{Co}

\begin{proof} This follows from the Lemmas \ref{hshjsjk} and \ref{G.essentialsurj} and the fact that by definition functors in $\GCat$ respect finite limits and finite colimits.
\end{proof}

\begin{Th}\label{galois equivalence} The 2-category of finitely connected Galois categories is contravariantly 2-equivalent to the 2-category of profinite finitely connected groupoids.
\end{Th}

\begin{proof} This equivalence is given by associating to a profinite groupoid $\G$, the Galois category $\Hom_{\Cat}(\G,\Sets)$. On functors and natural transformations, the 2-functor is defined in the obvious way by composition. To show that this is a 2-equivalence, we only need to show that it's essentially surjective and full and faithful. Both of course in the 2-mathematical sense. Essential surjectivity is proven in Corollary \ref{G.essentialsurj}. 
\newline
\underline{Full and Faithful}: Let $\G$ and $\mathfrak{H}$ be profinite and finitely connected groupoids. We have $\G\cong\coprod_{i\in I}\G_i$ and $\mathfrak{H}\cong \coprod_{j\in J}\mathfrak{H}_j$, where the $\G_i$-s and $\mathfrak{H}_j$-s are profinite groupoids with one object. Hence 
\begin{eqnarray*} \Hom_\Cat(\G,\mathfrak{H})&\cong&\Hom_{\Cat}(\coprod_{j\in J}\G_i,\coprod_{j\in J}\mathfrak{H}_j) \\
					     &\cong&\prod_{i\in I}\Hom_\Cat(\G_i,\coprod_{j\in J}\mathfrak{H}_j) \\
					    &\cong& \prod_{i\in I}\coprod_{j\in J}\Hom_\Cat(\G,\mathfrak{H}_j).
\end{eqnarray*}
The last equivalence comes from the fact that the $\G_i$-s and $\mathfrak{H}_j$-s have a single object. Hence any functor $\G_i$ can only go to a single $\mathfrak{H}_j$. From \cite[Corolarry 6.2. p.111]{SGAI} we get  that $\Hom_\Cat(\G_i,\mathfrak{H}_j)\cong\Hom_\GCat(\Hom_\Cat(\mathfrak{H}_j,\FSets),\Hom_\Cat(\G_i,\FSets)).$ On the other hand z
\begin{eqnarray*} & & \Hom_\GCat(\Hom_\Cat(\mathfrak{H},\FSets),\Hom_\Cat(\G,\FSets))\\ 
		          &\cong&\Hom_\GCat(\Hom_\Cat(\coprod_{j\in J}\mathfrak{H}_j,\FSets),\Hom_\Cat(\coprod_{i\in I}\G_i,\FSets)) \\ 
	                    &\cong&\Hom_\GCat(\Hom_\Cat(\coprod_{j\in J}\mathfrak{H}_j,\FSets),\prod_{i\in I}\Hom_\Cat(\G,\FSets)) \\
		         &\cong&\prod_{i\in I}\Hom_\GCat(\Hom_\Cat(\coprod_{j\in J}\mathfrak{H}_j,\FSets),\Hom_\Cat(\G,\FSets)).
\end{eqnarray*}
Using  Corollary \ref{G.faithfullyflat} we get the desired result.
\end{proof}

\begin{Co}\label{eq} Let $\G\rightarrow \mathfrak{H}$ be a functor between finitely connected, profinite groupoids. Assume that the induced functor $\Hom_{\Cat}(\mathfrak{H},\FSets)\rightarrow \Hom_{\Cat}(\G,\FSets)$ is an equivalence of categories. Then $\G\rightarrow \mathfrak{H}$ was already an equivalence of categories.
\end{Co}

\begin{proof} From Theorem \ref{galois equivalence} we get that 
$$\mathsf{Eq}(\G,\mathfrak{H})\cong \mathsf{Eq}(\G\text{-}\FSets,\mathfrak{H}\text{-}\FSets)$$
where $\mathsf{Eq}$ denotes the category of equivalences. This immediately implies the result.
\end{proof}

Unfortunately, we can not take general 2-limits with values in $\GCat$. So we can not talk about stacks with values in the 2-category of Galois categories for a general site. However, it follows from Proposition \ref{galois2limit} below, that if we only consider sites where every covering can be replaced by a finite one, we can avoid that problem. 

\begin{De} We call a site \emph{finitely coverable}\index{Site! Finitely coverable} if for every covering $\{U_i\rightarrow U\}_{i\in I}$ there exists a refinement $\{V_j\rightarrow U\}_{j\in J}$ such that $J$ is a finite set.
\end{De}

A stack with values in the 2-category of Galois categories will be referred to as a \emph{Galois stack}\index{Galois Category! Stack}. Similarly for a prestack. If we have two 2-functors $\F,\G:X\rightarrow \GCat$ from a finitely coverable site $X$ with values in Galois categories, then we will call a morphism between them that respects the structures a \emph{Galois transformation}\index{Galois Category! Transformation}, even though it would technically be a Galois functor. 

\begin{Rem} It should be noted that for Galois 2-functors $\F$ and $\G$, when we write $\Hom(\F,\G)$, we automatically assume that these are Galois transformations (i.e. respect finite limits and finite colimits), not just morphisms between 2-functors. 
\end{Rem}

Let $X$ be a site and $\F:X\rightarrow \sf{Groupoids}$ a covariant 2-functor. Then we denote by $\F_S$ the contravariant 2-functor given by $U\mapsto Hom_{\Cat}(\F(U),\Sets)$. Now take two covariant 2-functors $\E,\F:X\rightarrow \sf{Groupoids}$ and $F:\E\rightarrow \F$ a natural transformation. Then it is clear that $F_S:\F_S\rightarrow \E_S$ is a Galois transformation. But indeed Theorem \nolinebreak \ref{galois equivalence} shows that the reverse is also true. Hence we have the following as well:

\begin{Co}\label{Galoisstackequivalence} Let $X$ be a site. Then the 2-category of fibred functors over $X$ with values in Galois categories, and morphisms and 2-morphisms preserving the Galois structure, is contravariantly equivalent to the 2-category of cofibered functors over $X$ with values in profinite groupoids.
\end{Co}

\subsection{Galois categories under stackification}

From Proposition \ref{direct} it follows that if a property is preserved by both 2-limits and filtered 2-colimits, it is preserved by stackification. Hence we have these following facts, which can be checked by routine methods for 2-limits and 2-colimits separately:

\begin{Pro}\label{galois2limit} Let $I$ be a finite category and $\F:I\rightarrow \GCat$ be a 2-functor from $I$ to the 2-category of Galois categories given by $i\mapsto \{\cF_{ij}:\cC_i\rightarrow\FSets\}_{j\in J_i}$. Then the 2-limit of $\F$ is again a Galois category and all the natural projections $2$-$\lim\limits_i\F_i\rightarrow \F_i$ preserve the Galois structure, i.e. are exact.
\end{Pro}

\begin{Le}\label{Galois} Let $X$ be a site and $\F:X\rightarrow \Cat$ be a 2-functor, $\G:X\rightarrow \Cat$ a stack in $\Cat$ and $\psi:\F\rightarrow \G$ a natural transformation. Assume that $\F$ has an associated stack $\hat{F}$ and let $\hat{\psi}:\hat{\F}\rightarrow \G$ be the associated natural transformation. Then the following hold:
\begin{itemize} 
\item[i)] If $\psi$ respects (finite) limits, so does $\hat{\psi}$;
\item[ii)] If $\psi$ respects (finite) colimits, so does $\hat{\psi}$.
\end{itemize}
\end{Le}

Using  Proposition \ref{galois2limit} we can talk about stacks with values in Galois categories (called Galois stacks). As such, using the universal definition, we still have the notion of an associated Galois stack, even if we don't prove that it always exists. We see then that the above can be reformulated in the following way: 

\begin{Co}\label{galois-fun-stackification} Let $\F:X\rightarrow \GCat$ be a 2-functor, $\G:X\rightarrow \GCat$ a Galois stack and $\psi:\F\rightarrow \G$ a Galois transformation. If $\F$ has an associated Galois stack $\hat{\F}$, then $\hat{\psi}:\hat{\F}\rightarrow \G$ is a Galois transformation as well.
\end{Co}

\section{The \'Etale Fundamental Groupoid}\label{etale} 

In this section we will state and prove the main result of this paper. We will prove that for any \'etale  map $Y\mapsto X$ of a noetherian scheme $X$, the association $\Pi_1:Y\mapsto\Pi_1(Y)$ is the 2-terminal costack.

That is to say, the \'etale fundamental groupoid is defined by the Seifert-van Kampen theorem, in complete analogue to the topological case, as proven in \cite{top.groupoid}. The proof however is different.

To prove that the 2-functor $\Pi_1$ is the 2-terminal costack over any noetherian scheme $X$ is equivalent to saying that it is the associated costack of the constant, covariant 2-functor taking the trivial groupoid as its value. Unfortunately we don't know much about costackification, so instead we will compose it with the functor $\Hom(-,\FSets)$, and hence get a Galois category. 

Since by the definition of a costack, after composing with $\Hom(-,\FSets)$ we will get a stack, we can use stackification to prove this theorem. This is also our main reason for studying properties preserved under 2-limits and filtered 2-colimits. 
\newline

Let $X$ be a noetherian scheme. We denote by $\FEC(X)$ the site of finite \'etale coverings of $X$. It is a well known result that the category of finite \'etale coverings is a finitely connected Galois category and that the following holds:

\begin{Th}\label{Thm.8.0.2} The 2-functor $\mathcal{_FE_C}:\FEC(X)\rightarrow \GCat$, where $\FEC$ denotes the category of finite \'etale coverings, given by $Y\mapsto\FEC(Y)$, forms a stack.
\end{Th}

\begin{De} Let $\{\mathcal{F}_j:\cC\rightarrow\FSets\}_{j\in J}$ be a finitely connected Galois category. Then we can define the \emph{fundamental groupoid}\index{Fundamental groupoid! Galois category} $\Pi_1(\cC)$ of our Galois category as follows: 
\begin{itemize} 
\item Objects of $\Pi_1(\cC)$ are the functors $\{\cF_j:\cC\rightarrow \FSets\}_{j\in J}$;
\item For $\cF_i$ and $\cF_j$ we define $\Hom_{\Pi_1(\cC)}(\cF_i,\cF_j):=\mathsf{Iso}_{\mathsf{Fnct}}(\cF_i,\cF_j)$ where $\mathsf{Iso}_{\mathsf{Fnct}}$ denotes the set of natural isomorphisms.
\end{itemize}
\end{De}

\begin{De} Let $X$ be a site. We define the \emph{\'etale fundamental groupoid}\index{Fundamental groupoid! \'Etale} $\Pi_1(X)$ of $X$ to be the fundamental groupoid of the Galois category $\mathcal{_FE_C}(X)$ of finite \'etale coverings of $X$.
\end{De}

Equivalently, Theorem \ref{Thm.8.0.2} can be stated as follows:

\begin{Th}\label{hom-stack} The 2-functor $\mathcal{_FE_C}:\FEC(X)\rightarrow \GCat$ given by $Y\mapsto\Pi_1(Y)$-$\FSets$, where $\Pi_1(Y)$ denotes the \'etale fundamental groupoid of $Y$, forms a stack.
\end{Th}

\begin{Th}[Seifert-van Kampen Theorem] Let $X$ be a noetherian scheme. Then the assignment $Y\mapsto \Pi_1(Y)$ defines a costack on the site of finite \'etale coverings of $X$.
\end{Th}

\begin{proof} For any covering $Z\in \Cov(Y)$ we have the functor $2$-$\colim(Z,\Pi_1)\rightarrow \Pi_1(Z)$, where $2$-$\colim(Z,\Pi_1)$ denotes the 2-colimit of 
$$\xymatrix{\Pi_1(Z\times_{X} Y)& \Pi_1(Z\times_{X} Y\times_{X} Y)\ar@<-.5ex>[l] \ar@<.5ex>[l]& \Pi_1(Z\times_{X} Y\times_{X}Y\times_{X} Y).\ar@<-.7ex>[l]\ar[l]\ar@<.7ex>[l]}$$ 
Hence we get the associated functor $\Pi_1(Z)$-$\FSets\rightarrow2$-$\colim(Z,\Pi_1)$-$\FSets$ where we denoted by $\G$-$\FSets$ the functor category $\Hom_{\Cat}(\G,\FSets)$. Since $\Hom(-,\FSets)$ is left exact, we have 
$$2\text{-}\colim(Z,\Pi_1)\text{-}\FSets\cong 2\text{-}\lim((Z,\Pi_1)\text{-}\FSets),$$ 
where $2$-$\lim((Z,\Pi_1)$-$\FSets)$ denotes the 2-limit of 
$$\xymatrix@=1.5em{\Pi_1(Z\times_X Y)\text{-}\FSets\ar@<-.5ex>[r] \ar@<.5ex>[r] & \Pi_1(Z\times_X Y\times_X Y)\text{-}\FSets\ar@<-.7ex>[r]\ar[r]\ar@<.7ex>[r] & \Pi_1(Z\times_X Y\times_X Y\times_X Y)\text{-}\FSets.}$$ 
By Theorem \ref{hom-stack} the functor $\Pi_1(Z)\text{-}\FSets\rightarrow 2$-$\lim((Z,\Pi_1)$-$\FSets)$ is an equivalence of categories. Hence from Corollary \ref{eq} $2$-$\colim(Z,\Pi_1)\rightarrow \Pi_1(Z)$ was an equivalence of categories as well. This proves the assertion.
\end{proof}

\begin{De} Let $\C$ be a 2-category. We say that $\mathfrak{T}$ is the \emph{2-terminal object}\index{Object! 2-Terminal} of $\C$, if for any other object $C\in \C$, $Hom_{\Cat}(C,\mathfrak{T})$ is equivalent to the 1-point category.
\end{De}

\begin{Th}\label{main} Let $X$ be a noetherian scheme. Then the assignment $U\mapsto\Pi_1(U)$, $U\in X$ is the 2-terminal costack over the site of \'etale coverings of $X$.
\end{Th}

To prove this theorem, recall the following lemma:

\begin{Le}\label{LCS} Consider the constant 2-functor $\mathfrak{s}:U\mapsto\FSets$, with morphisms in the contravariant way. Then the associated stack $\hat{\mathfrak{s}}$ is given by $\hat{\mathfrak{s}}(U)=LCS(U)$, where $LCS(U)$ denotes the category of locally constant sheaves on $U$.
\end{Le}

Let $A$ be a covariant 2-functor. Recall that we denoted by $A_S$ the contravariant 2-functor given by $U\mapsto Hom(A(U),\FSets)$.

\begin{proof}[Proof of Thm. \ref{main}] We have already shown that the assignment $U\rightarrow \Pi_1(U)$ forms a costack. Hence to prove this theorem, we essentially have to show that for every costack $C$ we have an essentially unique map $C\rightarrow \Pi_1$.
Denote by $P$ the constant, covariant assignment $U\mapsto \pt$. It is clear that we have a map $C\rightarrow P$ and hence a map $P_S\rightarrow C_S$ between Galois 2-functors, which is a Galois transformation as it's induced by a functor between groupoids. As shown in Lemma \ref{LCS}, the stackification of $P_S$ exists and is $U\mapsto LCS(U)$, where $LCS(U)$ denotes the category of locally constant sheaves on $U$. In the case of noetherian schemes, $LCS(U)$ is equivalent to $\Pi_1(U)$-$\FSets$, which is a Galois category. Since $C$ was a costack, $C_S$ is a Galois stack and hence the map $P_S\rightarrow C_S$ factors through $\Pi_{1S}$. Hence, using Corollary \ref{galois-fun-stackification}, we know that the map $\Pi_{1S}\rightarrow C_S$ is a Galois transformation. 

Hence by the uniqueness of the associated stack (see Definition \ref{uniquestack}) and Corollary \ref{Galoisstackequivalence} respectively, we have the following equivalences of categories: 
$$\Hom_{\mathsf{Gal}}(P_S,C_S)\cong\Hom_{\mathsf{Gal}}(\Pi_{1S},C_S)\cong \Hom_\Cat(C,\Pi_1).$$
Here $\Hom_\mathsf{Gal}$ denotes the category of Galois transformations (i.e. of natural transformations respecting finite limits and colimits). Uniqueness and existence now comes from the fact that we have precisely one exact functor $P_S\rightarrow C_S$. The last claim is true because $P_S$ is equivalent to $\FSets$ and hence a functor respecting finite colimits is defined by its value on the singleton $\star$, which has to map to the terminal object of $C_S$. 
\end{proof}

\end{document}